\documentclass[12pt]{article}
\usepackage{booktabs}
\usepackage{caption}
\usepackage{mathrsfs}
\usepackage{amsmath}
\usepackage{amsfonts,amsthm,amssymb,mathrsfs,bbding}
\usepackage{float}
\usepackage{graphicx}
\usepackage{enumerate}
\usepackage{tikz}
\usepackage{caption}
\usepackage{rotating}
\allowdisplaybreaks[4]
\usepackage{tabularx}
\usepackage{cite}
\evensidemargin=\oddsidemargin\topmargin=-1.5cm

\newtheorem{thm}{Theorem}[section]
\newtheorem{prob}{Problem}[section]

\newtheorem{lem}{Lemma}[section]
\newtheorem{cor::3.1}{Corollary}[section]

\renewcommand\proofname{\bf Proof}
\addtocounter{section}{0}
\begin{document}

\title{\bf A unified approach to the spectral radius, connectivity and edge-connectivity of graphs }
\author{Yu Wang$^a$,\  Dan Li$^a$, \ Huiqiu Lin$^{a,b}$\thanks{Corresponding author.~~Email address: huiqiulin@126.com (H. Lin)} \\[2mm]
\small\it $^a$College of Mathematics and System Science, Xinjiang
University, \\
\small\it   Urumqi 830017, PR China\\[1mm]
\small\it  $^b$Department of Mathematics, East China University of Science and Technology,  \\
\small\it   Shanghai 200237, PR China}
\date{}
\small{}
\maketitle
{\flushleft\large\bf Abstract}:
For two integers $r\geq 2$ and $h\geq 0$, the \emph{$h$-extra $r$-component connectivity} $\kappa^h_r(G)$ of a graph $G$ is defined to be the minimum size of a subset of vertices whose removal
disconnects $G$, and there are at least $r$ connected components in $G\!-\!S$ and each component has at least $h+1$ vertices. Denote by $\mathcal{G}_{n,\delta}^{\kappa_r^h}$ the set of graphs with $h$-extra $r$-component connectivity $\kappa^h_r(G)$ and minimum degree $\delta$. The following problem concerning spectral radius was proposed by Brualdi and Solheid [On the spectral radius of complementary acyclic matrices of zeros and one,
SIAM J. Algebra Discrete Methods 7 (1986) 265-272]: Given a set of graphs $\mathscr{S}$, find
an upper bound for the spectral radius of graphs in $\mathscr{S}$ and characterize the graphs in which the maximal spectral
radius is attained. 
We study this question for $\mathscr{S}=\mathcal{G}_{n,\delta}^{\kappa_r^h}$ where $r\geq 2$ and $h\geq 0$. Fan, Gu and Lin [$l$-connectivity, $l$-edge-connectivity and spectral radius of graphs, \emph{arXiv}:2309.05247] give the answer to $r\geq 2$ and $h=0$. In this paper, we solve this problem completely for $r\geq 2$ and $h\geq1$. Moreover, we also investigate analogous problems for the edge version. Our results can break the restriction of the extremum structure of the conditional connectivity.  This implies some previous results in connectivity and edge-connectivity.

\maketitle {\flushleft\textit{\bf Keywords}:
Spectral radius; $h$-extra $r$-component connectivity; $h$-extra $r$-component edge-connectivity.}

\section{Introduction}
\vspace{1ex}

We use $\kappa(G)$ and $\lambda(G)$ to represent the
\emph{connectivity} and \emph{edge-connectivity} of a graph $G$, respectively. The connectivity and its generalizations have been studied due to their impact on
the fault tolerance and diagnosability of the interconnection networks\cite{T.L,M.J,L.Z}.  In order to break the restriction of the extremum structure of the conditional
connectivity,
Li, Lan, Ning, Tian, Zhang and Zhu \cite{B.L} introduced the concept of \emph{$h$-extra $r$-component connectivity} of a graph, as an extension of the classic connectivity. For two integers $r\geq 2$ and $h\geq 0$, the $h$-extra $r$-component connectivity $\kappa^h_r(G)$ of a graph $G$ is defined to be the minimum size of a subset of vertices whose removal
disconnects $G$, and there are at least $r$ connected components in $G\!-\!S$ and each component has at least $h+1$ vertices.

For a graph $G$, let $A(G)$ denote the adjacency matrix of $G$ and $\lambda_i(G)$ denote the $i$th largest eigenvalue of $A(G)$.
Particularly, the largest eigenvalue of $A(G)$, denoted by $\rho(G)$, is called the \textit{spectral radius} of $G$.

Connectivity has been well studied from eigenvalue perspectives in the past several decades. In 1973, Fiedler \cite{M.F} provided that $\kappa(G)\geq d-\lambda_2(G)$ in a $d$-regular graph.
Later, Krivelevich and Sudakov \cite{M.K} improved this result to $\kappa(G) \geq d-36 \lambda_2^2 / d$, where $d \leq n/2$.
The following problem concerning spectral radius was proposed by Brualdi and Solheid \cite{R.A}: Given a set of graphs $\mathscr{S}$, find
an upper bound for the spectral radius of graphs in $\mathscr{S}$ and characterize the graphs in which the maximal spectral
radius is attained.
By adding minimum degree $\delta(G)$ of a graph $G$ as a parameter, Lu and Lin \cite{Lu} proved that the maximum spectral radius is obtained uniquely at $K_k\vee(K_{\delta-k+1}\cup K_{n-\delta-1})$ among all graphs with $\kappa(G)\leq k \leq \delta(G)$, which extend the results of Li, Shiu and Chan \cite{J.L}, as well as Ye, Fan and Wang \cite{M.L.}. For more results, we can refer to references \cite{A.A,S.M,Z.M,Lin,S.O.}. Notice that $\kappa^0_2(G)=\kappa(G)$.
Thus, we naturally propose the following problem.

\begin{prob}\label{prob1}
What are the corresponding extremal graphs in $\mathcal{G}_{n,\delta}^{\kappa_r^h}$ with the maximum spectral radius for $r\geq 2$ and $h\geq 0$, where $\mathcal{G}_{n,\delta}^{\kappa_r^h}$ is the set of graphs with $h$-extra $r$-component connectivity $\kappa^h_r$ and minimum degree $\delta$?
\end{prob}

Around this problem, Fan, Gu and Lin\cite{D.F.} done the preliminarily work and gave the answer to $r\geq 2$ and $h=0$. In this paper, we further study the corresponding extremal graphs in $\mathcal{G}_{n,\delta}^{\kappa_r^h}$ for $r\geq 2$ and $h\geq 1$.
It is well-known that $\kappa(G)\geq \delta(G)$. Nevertheless, there is no such bound between $\delta$ and $\kappa^h_r$ for $r \geq 3$.
Denote by  $\vee$ and  $\cup$ the join and union products, respectively. For $\delta\leq \kappa^h_r$, let $G^{\kappa^h_r,\delta}_n$ be the graph obtained from
$K_{\kappa^h_r}\vee(K_{n-\kappa^h_r-(r-1)(h+1)}\cup (r-2)K_{h+1}\cup K_h)\cup K_1$ by adding $\delta-1$ edges between  $K_1$ and $K_{\kappa^h_r}$, and then adding one edge between $K_1$ and $K_h$. For
$\kappa^h_r+1\leq\delta< \kappa^h_r+h$, let $G^{\kappa^h_r,\delta}_n$ be the graph obtained from
$K_{\kappa^h_r}\vee(K_{n-\kappa^h_r-(r-1)(h+1)}\cup (r-2)K_{h+1}\cup K_h\cup K_1)$ by adding $\delta-\kappa^h_r$ edges
between $K_1$ and $K_h$. For $\delta\geq \kappa^h_r+h$, let $G^{\kappa^h_r,\delta}_n \cong
K_{\kappa^h_r} \vee(K_{n-\kappa^h_r-(r-1)(\delta-\kappa^h_r+1)} \cup(r-1) K_{\delta-\kappa^h_r+1})$.

\begin{thm}\label{thm::1.2}
Suppose that $r\geq2$ and $h\geq1$ are two integers. Let $G\in \mathcal{G}_{n,\delta}^{\kappa_r^h}$ where $n\geq \kappa^h_r+r(h+1)$. Then
$$\rho(G)\leq \rho(G^{\kappa^h_r,\delta}_n),$$
 with equality if and only if $G\cong G^{\kappa^h_r,\delta}_n.$
\end{thm}

 Analogous to the $h$-extra $r$-component
connectivity,
Yang, Zhang and Meng\cite{Y.Y.} proposed the concept of \emph{$h$-extra $r$-component
edge-connectivity} of $G$. For two integers $r\geq 2$ and $h\geq 0$, the $h$-extra $r$-component edge-connectivity $\lambda^h_r(G)$ of a graph $G$ is defined to be the minimum size of a subset of edges whose removal
disconnects $G$, and there are at least $r$ connected components in $G\!-\!S$ and each component has at least $h+1$ vertices.

Much work has been done on the relationship between edge-connectivity and eigenvalues. In 2004, Chandran\cite{S.L.} showed that $\lambda(G)=d$ if $G$ is a $n$-vertex
$d$-regular graph with $\lambda_2(G)<d-1-\frac{d}{n-d}$. Afterwards, Cioab\u{a}
\cite{S.M.} proved that $\lambda(G)\geq k$ if $\lambda_2(G)<d-\frac{2(\lambda-1)}{d+1}$ in a regular graph. Gu, Lai, Li and Yao\cite{Gu} further extended the result of Cioab\u{a} to general graphs, and showed that $\lambda(G)\geq k$ if $\lambda_2(G)<\delta-\frac{2(\lambda-1)}{\delta+1}$ in a graph $G$ with minimum degree $\delta\geq 2\lambda\geq 4$.
Ning, Lu and Wang \cite{W.Ning} started to study edge connectivity of a graph in view of spectral radius, and they conjectured that the graph with the maximum spectral radius among all graphs of order $n$ with minimum degree $\delta$ and edge-connectivity $\lambda$ is  $F^{\lambda}_{n,\delta+1}$, where $F^{\lambda}_{n,\delta+1}$ is the graph obtained from $K_{\delta+1}\cup K_{n-\delta-1}$
by adding $\lambda$ edges joining a vertex in $K_{\delta+1}$ and $\lambda$ vertices in $K_{n-\delta-1}$. Subsequently, Fan, Gu and Lin \cite{D.D.} confirmed this conjecture for $n\geq 2\delta+4$. Furthermore, Fan, Gu and Lin\cite{D.F.} and Wang, Lin and Tian\cite{Y.W} respectively characterized the extremal graphs with the maximum spectral radius among all connected graphs $G$ with $\lambda_{r}^{0}(G)$ and $\lambda_{2}^{1}(G)$. Denote by $\mathcal{B}_{n,\delta}^{\lambda_r^h}$ the set of graphs with $h$-extra $r$-component edge-connectivity $\lambda^h_r$ and minimum degree $\delta$. Motivated by above results, in this paper, we will consider the extremal graphs with the maximum spectral radius in $\mathcal{B}_{n,\delta}^{\lambda_r^h}$, where
$r\geq 2$ and $h\geq \delta$.


Let $B^{\lambda^h_r,\delta}_n$ be the graph obtained from $(K_{n-(r-1)(h+1)-\lambda_r^h+r-2}\vee
K_{\lambda_r^h-r+2})\cup(r-2)K_{h+1}\cup (K_\delta\vee K_{h-\delta})\cup K_1$ by adding $\delta$ edges between $K_1$ and $K_\delta$, and $\lambda_r^h-r+2$ edges between some vertex in a copy of $K_{h+1}$ and all vertices in $K_{\lambda_r^h-r+2}$, and finally adding $r-2$ edges between a vertex in $K_{\lambda_r^h-r+2}$
and a vertex in each remaining $(r-3)$ copy of $K_{h+1}$ and $K_\delta$, respectively (see Fig. 1).
\begin{figure}[http]
\centering
\begin{tikzpicture}[x=1.00mm, y=1.00mm, inner xsep=0pt, inner ysep=0pt, outer xsep=0pt, outer ysep=0pt]
\path[line width=0mm] (-2.00,-2.00) rectangle +(115.11,55.34);
\definecolor{L}{rgb}{0,0,0}
\path[line width=0.24mm, draw=L] (51.45,43.74) ellipse (24.88mm and 4.91mm);
\path[line width=0.24mm, draw=L] (25.91,18.24) ellipse (3.70mm and 12.73mm);
\definecolor{L}{rgb}{0,0,0.502}
\definecolor{F}{rgb}{0,0,0.502}
\path[line width=0.30mm, draw=L, fill=F] (81.99,22.12) circle (0.40mm);
\path[line width=0.30mm, draw=L, fill=F] (31.97,43.52) circle (0.50mm);
\definecolor{L}{rgb}{0.863,0.0784,0.235}
\definecolor{F}{rgb}{0.863,0.0784,0.235}
\path[line width=0.30mm, draw=L, fill=F] (34.81,43.52) circle (0.50mm);
\path[line width=0.30mm, draw=L, fill=F] (37.66,43.52) circle (0.20mm);
\path[line width=0.30mm, draw=L, fill=F] (40.50,43.52) circle (0.20mm);
\path[line width=0.30mm, draw=L, fill=F] (43.34,43.52) circle (0.20mm);
\path[line width=0.30mm, draw=L, fill=F] (25.78,26.74) circle (0.50mm);
\definecolor{L}{rgb}{0,0,0}
\definecolor{F}{rgb}{0,0,0}
\path[line width=0.30mm, draw=L, fill=F] (45.40,19.63) circle (0.40mm);
\path[line width=0.30mm, draw=L, fill=F] (48.25,19.63) circle (0.40mm);
\path[line width=0.30mm, draw=L, fill=F] (50.95,19.63) circle (0.40mm);
\path[line width=0.24mm, draw=L] (38.09,18.04) ellipse (3.70mm and 12.73mm);
\definecolor{L}{rgb}{0,0,0.502}
\definecolor{F}{rgb}{0,0,0.502}
\path[line width=0.30mm, draw=L, fill=F] (37.85,26.68) circle (0.50mm);
\definecolor{L}{rgb}{0,0,0}
\path[line width=0.24mm, draw=L] (72.22,18.28) ellipse (3.70mm and 12.73mm);
\definecolor{L}{rgb}{0,0,0.502}
\path[line width=0.30mm, draw=L, fill=F] (72.07,26.74) circle (0.50mm);
\definecolor{L}{rgb}{0.863,0.0784,0.235}
\definecolor{F}{rgb}{0.863,0.0784,0.235}
\path[line width=0.30mm, draw=L, fill=F] (46.19,43.52) circle (0.50mm);
\definecolor{L}{rgb}{0,0,0}
\definecolor{F}{rgb}{0,0,0}
\path[line width=0.30mm, draw=L, fill=F] (53.06,43.52) circle (0.50mm);
\path[line width=0.30mm, draw=L, fill=F] (56.90,43.63) circle (0.50mm);
\path[line width=0.30mm, draw=L, fill=F] (60.20,43.52) circle (0.20mm);
\path[line width=0.30mm, draw=L, fill=F] (63.04,43.52) circle (0.20mm);
\path[line width=0.30mm, draw=L, fill=F] (65.88,43.52) circle (0.20mm);
\path[line width=0.30mm, draw=L, fill=F] (69.67,43.49) circle (0.50mm);
\path[line width=0.30mm, draw=L, fill=F] (25.64,19.21) circle (0.20mm);
\path[line width=0.30mm, draw=L, fill=F] (25.78,23.05) circle (0.50mm);
\definecolor{L}{rgb}{0,0,0.502}
\definecolor{F}{rgb}{0,0,0.502}
\path[line width=0.30mm, draw=L, fill=F] (71.90,22.43) circle (0.20mm);
\definecolor{L}{rgb}{0,0,0}
\definecolor{F}{rgb}{0,0,0}
\path[line width=0.30mm, draw=L, fill=F] (25.64,16.36) circle (0.20mm);
\path[line width=0.30mm, draw=L, fill=F] (25.62,13.38) circle (0.20mm);
\path[line width=0.30mm, draw=L, fill=F] (25.68,9.82) circle (0.50mm);
\definecolor{L}{rgb}{0,0,0.502}
\definecolor{F}{rgb}{0,0,0.502}
\path[line width=0.30mm, draw=L, fill=F] (71.90,20.46) circle (0.20mm);
\path[line width=0.30mm, draw=L, fill=F] (71.90,18.88) circle (0.20mm);
\path[line width=0.30mm, draw=L, fill=F] (72.07,16.79) circle (0.50mm);
\definecolor{L}{rgb}{0,0,0}
\definecolor{F}{rgb}{0,0,0}
\path[line width=0.30mm, draw=L, fill=F] (72.22,14.37) circle (0.20mm);
\path[line width=0.30mm, draw=L, fill=F] (72.22,12.38) circle (0.20mm);
\path[line width=0.30mm, draw=L, fill=F] (72.07,10.11) circle (0.20mm);
\path[line width=0.30mm, draw=L, fill=F] (72.07,7.40) circle (0.50mm);
\definecolor{L}{rgb}{0.863,0.0784,0.235}
\path[line width=0.24mm, draw=L] (25.78,27.40) -- (31.80,42.96);
\path[line width=0.24mm, draw=L] (25.69,27.08) -- (34.88,43.40);
\path[line width=0.24mm, draw=L] (25.39,26.68) -- (46.10,43.36);
\definecolor{L}{rgb}{0,0,0.502}
\path[line width=0.24mm, draw=L] (31.75,43.36) -- (37.95,26.58);
\path[line width=0.24mm, draw=L] (32.02,43.23) -- (72.17,26.64);
\path[line width=0.24mm, draw=L] (82.04,22.43) -- (72.63,26.68);
\path[line width=0.24mm, draw=L] (72.23,16.63) -- (81.91,22.04);
\definecolor{L}{rgb}{0,0,0}
\path[line width=0.30mm, draw=L, fill=F] (37.70,19.21) circle (0.20mm);
\path[line width=0.30mm, draw=L, fill=F] (37.85,23.05) circle (0.40mm);
\path[line width=0.30mm, draw=L, fill=F] (37.70,16.36) circle (0.20mm);
\path[line width=0.30mm, draw=L, fill=F] (37.69,13.38) circle (0.20mm);
\path[line width=0.30mm, draw=L, fill=F] (37.75,9.82) circle (0.50mm);
\draw(68.48,48.10) node[anchor=base west]{\fontsize{10.33}{12.39}\selectfont $K_{n\!-\!(\!r-\!1)(h\!+\!1)}$};
\draw(22.17,1.72) node[anchor=base west]{\fontsize{10.33}{12.39}\selectfont $K_{h+1}$};
\draw(34.33,1.51) node[anchor=base west]{\fontsize{10.33}{12.39}\selectfont $K_{h+1}$};
\draw(70.42,1.72) node[anchor=base west]{\fontsize{10.33}{12.39}\selectfont $K_{h}$};
\draw(82.96,22.43) node[anchor=base west]{\fontsize{10.33}{12.39}\selectfont $K_{1}$};
\definecolor{L}{rgb}{0.863,0.0784,0.235}
\path[line width=0.18mm, draw=L] (38.93,43.62) ellipse (11.01mm and 2.06mm);
\definecolor{L}{rgb}{0,0,0}
\path[line width=0.30mm, draw=L] (0.00,0.00) ellipse (0.00mm and 0.00mm);
\definecolor{L}{rgb}{0,0,0.502}
\path[line width=0.18mm, draw=L] (72.13,22.06) ellipse (2.51mm and 6.74mm);
\draw(22.67,48.50) node[anchor=base west]{\fontsize{10.33}{12.39}\selectfont $K_{\lambda_r^h\!-\!r\!+\!2}$};
\draw(75.59,28.26) node[anchor=base west]{\fontsize{10.33}{12.39}\selectfont $K_\delta$};
\definecolor{L}{rgb}{0,0,0}
\path[line width=0.24mm, draw=L] (59.15,18.04) ellipse (3.70mm and 12.73mm);
\definecolor{L}{rgb}{0,0,0.502}
\definecolor{F}{rgb}{0,0,0.502}
\path[line width=0.30mm, draw=L, fill=F] (58.91,26.68) circle (0.50mm);
\definecolor{L}{rgb}{0,0,0}
\definecolor{F}{rgb}{0,0,0}
\path[line width=0.30mm, draw=L, fill=F] (58.76,19.21) circle (0.20mm);
\path[line width=0.30mm, draw=L, fill=F] (58.91,23.05) circle (0.50mm);
\path[line width=0.30mm, draw=L, fill=F] (58.76,16.36) circle (0.20mm);
\path[line width=0.30mm, draw=L, fill=F] (58.75,13.38) circle (0.20mm);
\path[line width=0.30mm, draw=L, fill=F] (58.81,9.82) circle (0.50mm);
\draw(55.40,1.51) node[anchor=base west]{\fontsize{10.33}{12.39}\selectfont $K_{h+1}$};
\definecolor{L}{rgb}{0,0,0.502}
\path[line width=0.24mm, draw=L] (31.89,43.10) -- (58.61,26.95);
\definecolor{F}{rgb}{0,0,0.502}
\path[line width=0.30mm, draw=L, fill=F] (72.03,24.54) circle (0.50mm);
\path[line width=0.24mm, draw=L] (82.29,22.09) -- (72.03,24.41);
\end{tikzpicture}%
\begin{center}\small{Fig. 1 : $B^{\lambda_r^h,\delta}_n$}\end{center}
\end{figure}
\begin{thm}\label{thm::1.3}
Suppose that $r\geq2$ and $h\geq 1$ are two integers. Let $G\in \mathcal{B}_{n,\delta}^{\lambda_r^h}$ where $h \geq\delta$ and $n\geq (\lambda_r^h+1)(h+1)^2$. Then
$$\rho(G)\leq \rho(B^{\lambda_r^h,\delta}_n),$$
 with equality if and only if $G\cong
B^{\lambda_r^h,\delta}_n$.
\end{thm}

\section{ Proof of Theorem \ref{thm::1.2}}
For any $v\in V(G)$, let $N_{G}(v)=\{u\mid  uv\in E(G)~\text{and}~u\in V(G)\}$ and $N_{G}[v]=N_{G}(v)\cup\{v\}$.

\begin{lem}[See \cite{H.L}]\label{lem::2.1}
Let $G$ be a connected graph, and let $u$, $v$ be two vertices of $G$. Suppose $\{v_1, v_2, \ldots, v_s\} \subseteq N_{G}(v)
\backslash N_{G}(u)(1 \leq s \leq d_G(v))$ and  $G^*$ is the graph obtained from $G$ by adding the edges $uv_i$ and deleting
the edges $vv_i$, where $1 \leq i \leq s$. If $x_u \geq x_v$, then $\rho(G)<\rho(G^*)$.
\end{lem}

\begin{lem}[See \cite{B.S.}]\label{lem::2.2}
Let $u$, $v$ be two distinct vertices of a connected graph $G$, and let $x$ be the Perron vector of $A(G)$.

\noindent (i) If $N_G(v)\backslash\{u\}\subset N_G(u)\backslash\{v\}$, then $x_u> x_v$.

\noindent (ii) If $N_G(v)\subseteq N_G[u]$ and $N_G(u)\subseteq N_G[v]$, then $x_u=x_v$.
\end{lem}

\begin{lem}[See \cite{R.A.}]\label{lem::2.3}
Let $G$ be a connected graph, and let $G'$ be a proper subgraph of $G$. Then $\rho(G')<\rho(G)$.
\end{lem}

\renewcommand\proofname{\bf Proof of Theorem \ref{thm::1.2}}
\begin{proof}
Suppose that $G'\in\mathcal{G}_{n,\delta}^{\kappa_r^h}$ is a graph that attains the maximum spectral radius.  Then
\begin{equation}\label{equ::1}
\begin{aligned}
\rho(G)\leq \rho(G').
\end{aligned}
\end{equation}
Since $n\geq \kappa^h_r+r(h+1)$, by the definition of $h$-extra $r$-component connectivity, there exists some nonempty
subset $S\subseteq V(G')$ with $|S|=\kappa^h_r$ such that $G'-S$ contains at least $r$ components and each
component has at least $h+1$ vertices. Let $G'-S=C_1\cup C_2\cup\cdots\cup C_q$ where $q\geq r$, and let 
$|V(C_i)|=n_i\geq 2$ for $1\leq i\leq q$. Then $q=r$ by Lemma \ref{lem::2.3}. Choose a vertex $u\in V(G')$
such that $d_{G'}(u)=\delta$. Suppose that $d_S(u)=t$, $P_i=N_{G'}(u)\cap V(C_i)$ and  $|P_i|=p_i$ for $1\leq i\leq r$. Thus,
$\sum_{i=1}^{r}p_i=\delta-t$. Assume that $x$ is the Perron vector of $A(G')$. By symmetry,  $x_v=x_i$ for $v\in P_i$ and 
$x_v={x_i}'$ for $v\in V(C_i)\backslash P_i$, where $1\leq i\leq r$. Without loss of generality, assume that
$x_1=\max\{x_i\mid 1\leq i\leq r\}$. Now, we divide the proof into the following three cases.

{\flushleft\bf{Case 1.}} $\delta\leq \kappa^h_r$.

Due to the uncertainty of the minimum degree vertex $u$, we first present the following  Claim.

{\flushleft\bf Claim 1.} $u\notin S$.

If not, let $u\in S$. By the maximality of $\rho(G')$, $G'-u\cong K_{\kappa^h_r-1}\vee(K_{n_1}\cup
K_{n_2}\cup\cdots\cup K_{n_r})$. We first assert that $N_{G'}(u)\nsubseteq V(S\cup C_i)$ and $N_{G'}(u)\nsubseteq V(C_i)$ for
$1\leq i\leq r$. Otherwise, let $N_{G'}(u)\subseteq V(S\cup C_i)$ or $N_{G'}(u)\subseteq V(C_i)$ for $1\leq i\leq r$.
Set $S'=S- \{u\}$. Obviously,
$\kappa_r^h=|S'|=|S|-1$, which contradicts $\kappa_r^h=|S|$. This implies that 
$d_S(u)\neq \delta-1$. We next assert that
$d_S(u)=\delta$ for $\delta\leq \kappa_r^h-1$ and $d_S(u)=\delta-2$ for 
$\delta= \kappa_r^h$. Otherwise, suppose that
$d_S(u)\leq\delta-2$ for $\delta\leq \kappa_r^h-1$ or $d_S(u)\leq \delta-3$ for 
$\delta= \kappa_r^h$. Then $|N_{G'}(u)\cap V(G'\backslash S)|\geq 2$. Let $N_S(u)=\{v_1,v_2,\ldots,v_t\}$ and $V(S)\backslash
N_S[u]=\{v_{t+1},v_{t+2},\ldots,v_{\kappa^h_r-1}\}$. Before proving that $d_S(u)\leq\delta-2$ for $\delta\leq \kappa_r^h-1$ or $d_S(u)\leq \delta-3$ for $\delta= \kappa_r^h$, let's first prove that $x_{v_{t+1}}\geq x_1$. Note that 
$x_{v_i}=x_{v_1}$ for $2 \leq i \leq t$ and
$x_{v_i}=x_{v_{t+1}}$ for $t+2 \leq i \leq \kappa^h_r-1$. Then
\begin{align}
\rho(G') x_i & =t x_{v_1}+(\kappa^h_r-t-1) x_{v_{t+1}}+(p_i-1) x_i+(n_i-p_i) x_i^{\prime}+x_{u},\label{equ::2}\\
\rho(G') x_i^{\prime} & =t x_{v_1}+(\kappa^h_r-t-1) x_{v_{t+1}}+p_i x_i+(n_i-p_i-1) x_i^{\prime}, \label{equ::3}\\
\rho(G') x_{u} & =t x_{v_1}+\sum_{i=1}^r p_i x_i,\label{equ::4}\\
\rho(G') x_{v_{t+1}} & =t x_{v_1}+(\kappa^h_r-t-2) x_{v_{t+1}}+\sum_{i=1}^r p_i x_i+\sum_{i=1}^r(n_i-p_i)
x_i^{\prime},\label{equ::5}
\end{align}
where $1 \leq i \leq r$. From (\ref{equ::2})-(\ref{equ::4}), we deduce that
\begin{align*}
& \rho(G')\Big (\sum_{i=2}^r p_i x_i+\sum_{i=2}^r(n_i-p_i) x_i^{\prime}-x_{u}\Big) \\
= & \Big(\sum_{i=2}^r n_i-1\Big) t x_{v_1}+\sum_{i=2}^r n_i(\kappa^h_r-t-1) x_{v_{t+1}}+\sum_{i=2}^r(n_i-1) p_i x_i+ \\
& \sum_{i=2}^r(n_i-1)(n_i-p_i) x_i^{\prime}+\sum_{i=2}^r p_i x_{u}-\sum_{i=1}^r p_i x_i \\
\geq & \big((r-1)(h+1)-1\big)t x_{v_1}+(r-1)(h+1)(\kappa^h_r-t-1) x_{v_{t+1}}- p_1 x_1\\
&\quad(\text{since }n_i \geq h+1 \text { and } 0 \leq p_i \leq n_i \text { for } 2 \leq i \leq r) \\
> & (h+1)(\kappa^h_r-t-1) x_{v_{t+1}}-(\delta-t-1) x_1 \\
&\quad(\text {since } r \geq 2,~\sum_{j=1}^r p_j=\delta-t \text { and } x_1 \geq x_i \text { for } 2 \leq i \leq r) \\
> & (\delta-t-1)(x_{v_{t+1}}-x_1) \quad(\text {since } h \geq 1 \text { and } \kappa^h_r \geq \delta).
\end{align*}
Combining this with (\ref{equ::2}) and (\ref{equ::5}), we have
$$
(\rho(G')+1)(x_{v_{t+1}}-x_1)=\sum_{i=2}^r p_i x_i+\sum_{i=2}^r(n_i-p_i) x_i^{\prime}-x_{u} \geq
\frac{(\delta-t-1)(x_{v_{t+1}}-x_1)}{\rho(G')},
$$
and hence
$$
\frac{(\rho^2(G')+\rho(G')-\delta+t+1)(x_{v_{t+1}}-x_1)}{\rho(G')} \geq 0.
$$
Notice that $G'$ contains $K_{\kappa^h_r-1+n_i}$ as a proper subgraph, where $1 \leq i \leq r$. Thus, $\rho(G')>$
$\rho(K_{\kappa^h_r-1+n_i})=\kappa^h_r+n_i-2 \geq \delta$ due to $n_i \geq 2$ for $1 \leq i \leq r$ and $\kappa^h_r \geq\delta$. Combining this with 
$t \geq 0$, we can conclude that $\rho^2(G')+\rho(G')-\delta+t>0$. This suggests that
$x_{v_{t+1}} \geq x_1$.

If $\delta\leq \kappa^h_r-1$, let $G_1=G'-\{u v \mid v \in P_i, 1 \leq i \leq r\}+\{u v_j \mid t+1 \leq j \leq \delta\}$.
According to Lemma \ref{lem::2.1},
\begin{equation}\label{equ::6}
\begin{aligned}
 \rho(G_1)>\rho(G').
\end{aligned}
\end{equation}
Furthermore, we take two different vertices $z,w\in V(C_r)$. Clearly, $x_z>x_u$ due to $N_{G_1}(u)\subseteq N_{G_1}(z)$ and Lemma
\ref{lem::2.2}. Note that $x_v=x_i'$ for $v\in V(C_i)$, $x_{v_i}=x_{v_1}$ for $2\leq i\leq \delta$ and
$x_{v_j}=x_{v_{\delta+1}}$ for $\delta+2\leq j\leq \kappa_r^h-1.$ Then
\begin{align}
\rho(G_1) x_{v_{1}}&=(\delta-1)x_{v_1}+(\kappa_r^h-\delta-1)x_{v_{\delta+1}}+\sum_{i=1}^{r}n_ix_i'+x_u,\label{equ::7}\\
\rho(G_1) x_i'&=\delta x_{v_1}+(\kappa_r^h-\delta-1)x_{v_{\delta+1}}+(n_i-1)x_i'.\label{equ::8}
\end{align}
Suppose that $E_1=\{zv\mid v\in V(G_1\backslash (S\cup C_r))\}+\{uw\}$ and  $E_2=\{uv_{\delta}\}$. Let $G_2=G_1+E_1-E_2$.
Combining this with (\ref{equ::7}) and (\ref{equ::8}),
\begin{align*}
&~~~\rho(G_2)-\rho(G_1)\geq x^T(A(G_2)-A(G_1))x \\
&=2\Big(\sum_{zv\in E_1}x_zx_v+x_ux_w-x_ux_{v_{\delta}}\Big)\\
&\geq 2\Big(\sum_{zv\in E_1}x_zx_v+x_zx_w-x_zx_{v_{\delta}}\Big)\quad(\text {since } x_{w} < x_{v_{\delta}} \text { and }
x_z>x_u) \\
&=2x_z\Big(\sum_{i=1}^{r-1}n_ix_i'+x_r'-x_{v_1}\Big)\\
&=\frac{2x_z}{\rho(G_1)}\Big(\sum_{i=1}^{r-1}n_i(\delta
x_{v_1}+(\kappa_r^h-\delta-1)x_{\delta+1}+(n_i-1)x_i')+(n_r-1)x_r'+\\
&~~\delta
x_{v_1}+(\kappa_r^h-\delta-1)x_{v_{\delta+1}}-\big((\delta-1)x_{v_1}+(\kappa_r^h-\delta-1)x_{v_{\delta+1}}+\sum_{i=1}^{r}n_ix_i'+x_u\big)\Big)\\
&\geq\frac{2x_z}{\rho(G_1)}\Big(3\delta
x_{v_1}+\sum_{i=1}^{r-1}n_ix_i'+(n_r-1)x_r'-(\delta-1)x_{v_1}-\sum_{i=1}^{r}n_ix_i'-x_u\Big)\\
&=\frac{2x_z}{\rho(G_1)}\Big((2\delta+1)x_{v_1}-x_r'-x_u\Big)\\
&> 0\quad(\text {since } x_{v_1} > x_r' \text { and } x_{v_1} > x_u).
\end{align*}
It follows that $\rho(G_2)>\rho(G_1)$.  Let $S^{\prime}=S-u+z$. Then 
$G_2-u\cong K_{\kappa^h_r}\vee(K_{n_1}\cup K_{n_2}\cup\cdots\cup K_{n_r-1})$, and hence $G_2$ is connected. Note that 
$S^{\prime}$ is the subset of $V(G_2)$ with minimum cardinality such that 
$G_2-S^{\prime}$ contains at least $r$ connected components, and each component has at least $h+1$vertices. 
Combining this with $\delta(G_2)=d_{G_2}(u)=\delta$ and  $|S^{\prime}|=\kappa^h_r$, we have $G_2\in \mathcal{G}_{n,\delta}^{\kappa_r^h}$. Furthermore, we obtain that $\rho(G_2)>\rho(G')$ by (\ref{equ::6}),
a contradiction. Thus, $u\notin S$ for $\delta\leq \kappa^h_r-1$.

If $\delta =\kappa^h_r$, since $d_S(u)\leq \kappa_r^h-2$ and $x_1=\max\{x_i\mid 1\leq i\leq r\}$, we have $p_1\geq1$. Let $z\in P_1$ and $G_3=G'-\{u v \mid v \in P_i\backslash\{z\}, 1 \leq i \leq r\}+\{u v_j \mid t+1 \leq j \leq
\kappa_r^h-1\}$. Recall that
$x_{v_{t+1}}>x_1$.  Thus, according to Lemma \ref{lem::2.1},
\begin{equation}\label{equ::9}
\begin{aligned}
\rho(G_3)>\rho(G').
\end{aligned}
\end{equation}
Take a vertex $w\in V(C_1)\backslash \{z\}$. Assume that $G_4=G_3+\{wv\mid v\in V(G\backslash S\cup C_1)\}$ and
$S^{\prime}=S-u+w$. Thus, $|S^{\prime}|=\kappa^h_r$, $G_4\in \mathcal{G}_{n,\delta}^{\kappa_r^h}$ and $\rho(G_4)>\rho(G_3)$ by Lemma \ref{lem::2.1}. Combining this with (\ref{equ::9}), we get
$\rho(G_4)>\rho(G^{\prime})$,
a contradiction. It follows that $u\notin S$ for $\delta= \kappa^h_r$. Completing the proof
of Claim 1.

Suppose that $u\in V(C_1)$ by Claim 1. Then $d_{C_1}(u)\geq 1$ because $n_1 \geq h+1\geq 2$ and $C_1$
is connected. This implies that $d_S(u)\leq d_{G_1}(u)-d_{C_1}(u)\leq \delta-1$.

{\flushleft\bf Claim 2.} $d_S(u)=\delta-1$.

If not, assume that $d_S(u)=t\leq\delta-2$. Thus, $d_{C_1}(u)=\delta-t\geq 2$. Let $N_S(u)=\{v_1, v_2, \ldots, v_t\}$, $S \backslash
N_S(u)=\{v_{t+1}, v_{t+2}, \ldots, v_{\kappa_r^h}\}$, $V(C_i)=\{u_i^1, u_i^2, \ldots, u_i^{n_i}\}$ and
$N_{C_1}(u)=\{u_1^1,u_1^2,\ldots,$ $u_1^{\delta-t}\}$. Again by the maximality of $\rho(G')$, it follows that $G'-u
\cong$ $K_{\kappa_r^h} \vee(K_{n_1-1} \cup K_{n_2} \cup \cdots \cup K_{n_r})$. Note that $x_v=x_1$ for $v\in N(u)\cap
V(C_1)$, $x_v=x_1'$ for $v\in V(C_1)\backslash N[u]$, $x_v=x_i'$ for $2\leq i\leq r$, $x_{v_i}=x_{v_1}$ for $2 \leq i \leq t$
and $x_{v_i}=x_{v_{t+1}}$ for $t+2 \leq i \leq \kappa_r^h$. Then
\begin{align*}
\rho(G') x_{v_{t+1}}&=t x_{v_1}+(\kappa_r^h-t-1) x_{v_{t+1}}+(\delta-t) x_1+(n_1-1-(\delta-t)) x_1^{\prime}+\sum_{i=2}^r
n_i x_i,\\
\rho(G') x_1 &=t x_{v_1}+(\kappa_r^h-t) x_{v_{t+1}}+(\delta-t-1) x_1+(n_1-1-(\delta-t)) x_1^{\prime}+x_u,\\
\rho(G') x_i' &=t x_{v_1}+(\kappa_r^h-t) x_{v_{t+1}}+(n_i-1)x_i',\\
\rho(G') x_u &=t x_{v_1}+(\delta-t)x_{v_1},
\end{align*}
where $2\leq i\leq r$. Hence,
\begin{align*}
 &\rho(G')(\rho(G')+1)(x_{v_{t+1}}-x_1)\\
= &\rho(G')(\sum_{i=2}^r n_i x_i'-x_u) \\
= & \sum_{i=2}^r n_i(\rho(G') x_i')-\rho(G') x_u \\
= & \sum_{i=2}^r n_i\big(t x_{v_1}+(\kappa_r^h-t) x_{v_{t+1}}+(n_i-1) x_i'\big)-(t x_{v_1}+(\delta-t) x_1) \\
> & ((r-1)(h+1)-1) t x_{v_1}+(r-1)(h+1)(\kappa_r^h-t) x_{v_{t+1}}-(\delta-t) x_1\\
&\quad(\text {since } n_i \geq h+1\geq 2 \text { for } 2 \leq i \leq r) \\
> & (h+1)(\kappa_r^h-t) x_{v_{t+1}}-(\delta-t) x_1\quad(\text {since } h \geq 1 \text { and } r \geq 2) \\
> & (\delta-t)(x_{v_{t+1}}-x_1)\quad(\text {since } \kappa_r^h \geq \delta),
\end{align*}
which indicates that $\big(\rho(G')(\rho(G')+1)-(\delta-t)\big)(x_{v_{t+1}}-x_1)>0$. Notice that $G'$ contains
$K_{\kappa^h_r+n_1-1}$ as a proper subgraph. Thus, $\rho(G')>\rho(K_{\kappa^h_r+n_1-1})=\kappa^h_r+n_1-2 \geq \delta$ due
to $n_1 \geq 2$ and $\kappa^h_r \geq \delta$, and hence $\rho(G')(\rho(G')+1)-\delta+t>0$. This suggests that
$x_{v_{t+1}}>x_1$. Take $G^*=G'-\{u u_1^i \mid 2 \leq i \leq \delta-t\}+\{u v_j \mid t+1 \leq j \leq \delta-1\}$. It is easy
to find that $G^*\in \mathcal{G}_{n,\delta}^{\kappa_r^h}$. So, $\rho(G^*)>\rho(G')$ by Lemma \ref{lem::2.1}, a contradiction. Thus, $d_S(u)=\delta-1$. Completing the proof
of Claim 2.

Recall that $|V(C_i)|=n_i\geq 2$ for $1\leq i\leq r$. Suppose that $n_1\geq n_2\geq \cdots\geq n_r$.

{\flushleft\bf Claim 3.} $u\in V(C_i)$ for some $n_i=n_r$, where $1\leq i\leq r$.

If $n_1=n_r$, the result is trivial. Next, we consider the case of $n_1>n_r$. We use proof by contradiction, assume that there exists $C_k$ such that $u\in V(C_k)$ for $n_k>n_r$, where $1\leq k\leq r-1$. Then $G'-u\cong K_{\kappa_r^h}\vee(K_{n_1}\cup\cdots \cup K_{n_k-1}\cup\cdots\cup
K_{n_r})$ by the maximality of $\rho(G')$. By symmetry, $x_{v_i}=x_{v_1}$ for $2 \leq i \leq \delta-1$ and
$x_{v_i}=x_{v_{\delta}}$ for $\delta+1 \leq i \leq \kappa_r^h$. Let $uu_k^1\in E(C_k)$, $x_v=x_k'$ for
$V(C_k)\backslash N_{C_k}[u]$, and $x_v=x_i'$ for $v\in{V(C_i)}$, where $1\leq i\leq r$ and $i\neq k$. Then
\begin{align}
\rho(G') x_k'&=(\delta-1)x_{v_1}+(\kappa_r^h-\delta+1)x_{v_\delta}+x_k+(n_k-2)x_k',\label{equ::10}\\
\rho(G') x_r'&=(\delta-1)x_{v_1}+(\kappa_r^h-\delta+1)x_{v_\delta}+(n_r-1)x_r',\label{equ::11}\\
\rho(G') x_u&=(\delta-1)x_{v_1}+x_{u_k^1}.\label{equ::12}
\end{align} From (\ref{equ::10}) and (\ref{equ::11}), we get
$x_k>x_k'=\frac{\rho(G')-n_r+1}{\rho(G')-n_k+2}x_r'+\frac{x_k}{\rho(G')-n_k+2}>x_r'$ due to $\rho(G')-n_r+1\geq
\rho(G')-n_k+2$ and $\rho(G')-n_k+2>0$. From (\ref{equ::11}) and (\ref{equ::12}), we obtain that $x_r'>x_u$ because
$\kappa_r^h-\delta+1\geq1$ and $x_{v_\delta}>x_{u_k^1}$ by Claim 2.

Suppose that $E_1=\{uu_r^1\}+\{u_r^{n_r}u_k^j\mid 1\leq j\leq n_k-1\}$, and $E_2=\{uu_k^1\}+\{u_r^{n_r}u_r^j\mid 1\leq j\leq
n_r-1\}$.
Let $G_1=G'+E_1-E_2$. Then $G_4\in \mathcal{G}_{n,\delta}^{\kappa_r^h}$ and
\begin{align*}
&~~~\rho(G_1)-\rho(G')\\
&\geq x^T(A(G_1)-A(G'))x \\
&=2x_ux_{u_r^k}+2x_{u_r^{n_r}}x_{u_k^1}+2(n_k-2)x_{u_r^{n_r}}x_{u_k^2}-2x_ux_{u_k^1}-2(n_r-1)x_{u_r^{n_r}}x_{u_r^1}\\
&=2x_ux_r'+2x_r'x_k+2(n_k-2)x_r'x_k'-2x_ux_k-2(n_r-1)x_r'x_r'\\
&>2x_k(x_r'-x_u)+2x_r'(n_r-1)(x_k'-x_r')\\
&>0\quad(\text {since } x_r'>x_u,~n_r\geq2\text{ and }x_k'> x_r').
\end{align*}
It follows that $\rho(G_1)>\rho(G')$,
a contradiction. Thus, $u\in V(C_i)$ for $n_i=n_r$, as desired.

Next, we shall prove that $G' \cong G_n^{\kappa^h_r, \delta}$ for $\delta\leq \kappa_r^h$. In fact, by Claims 1-3
and the maximality of $\rho(G')$, we can deduce that $G'-u \cong K_{\kappa^h_r} \vee(K_{n_1} \cup K_{n_2} \cup \cdots \cup
K_{n_{r}-1})$ for some integers $n_1 \geq n_2 \geq \cdots \geq n_r$, where $\sum_{i=1}^{r} n_i=n-\kappa^h_r$, $u\in V(C_r)$
and $d_S(u)=\delta-1$. If $n_i=h+1$ for all $2\leq i\leq r$, then the result holds. Suppose to the contrary that $n_i=h+1$. Thus, there exists some $n_k \geq h+2$ for
$2 \leq k \leq r$.  Note that $u\in V(C_r)$. Clearly, $x_u<x_r'<x_r<x_i'$ for $1\leq i\leq r-1$.
Since $$(\rho(G')-n_i+1)x_i'=(\delta-1)x_{v_1}+(\kappa^h_r-\delta+1)x_{v_\delta}$$ for any $1\leq i\leq r-1$ and
$x_1'\geq x_k'$ where $n_1 \geq n_k$. Take $w \in V(K_{n_k})$ and $G''=G'-\{w v \mid v \in V(K_{n_k}) \backslash\{w\}\}+\{w v \mid v \in V(K_{n_1})\}$. Then $G''\in \mathcal{G}_{n,\delta}^{\kappa_r^h}$ and $\rho(G'')>\rho(G')$ by Lemma \ref{lem::2.1}, a contradiction. It suggests that $G' \cong G_n^{\kappa_r^h, \delta}$, as desired.

{\flushleft\bf{Case 2.}} $\kappa^h_r+1\leq\delta<\kappa^h_r+h$.

Similar to Case 1, we first make the following Claim due to the uncertainty of the minimum degree vertex $u$.
{\flushleft\bf Claim 4.} $u\notin S$.

If not, $u\in S$. By the maximality of $\rho(G')$, we get $G'-u\cong K_{\kappa^h_r-1}\vee(K_{n_1}\cup
K_{n_2}\cup\cdots\cup K_{n_r})$.  By using a similar analysis as Case 1, we deduce that $N_{G'}(u)\nsubseteq V(S\cup C_i)$ and $N_{G'}(u)\nsubseteq V(C_i)$ for
$1\leq i\leq r.$ We assert that $d_S(u)=\kappa^h_r-1$. If not, $d_S(u)\leq \kappa^h_r-2$.  Then $N_{G'}(u)\cap
V(G'\backslash S)\geq 3$. Let
$V(C_i)=\{u_i^1,\ldots,u_i^{n_i}\}$ and $V(C_i)\cap N_{G'}(u)\{u_i^1,\ldots,u_i^{p_i}\}$ for $1\leq i\leq r$, and $N_{G'}(u)\cap V(G'-S)=\{u_1,\ldots,u_{\delta-t}\}.$
Similar to Claim 1 of Case 1, we have $x_{v_{t+1}}\geq x_1$. Recall that $N_{G'}(u)\nsubseteq V(S)\cup V(C_i)$ for $1\leq
i\leq r$. Without loss of generality, suppose that $u_1\in V(C_i)$ and $u_2\in V(C_j)$ for $1\leq i< j\leq r$. Let
$G_1=G'-\{uu_i\mid 3\leq i\leq \kappa^h_r-t+1\}+\{uv_j\mid t+1\leq j\leq \kappa^h_r-1\}$.
Then $G_1\in \mathcal{G}_{n,\delta}^{\kappa_r^h}$ and $\rho(G_1)>\rho(G')$ by Lemma \ref{lem::2.1}, a contradiction. This implies that $d_S(u)=\kappa^h_r-1$. Notice that $x_{v_i}=x_{v_1}$ for $2\leq i\leq \kappa^h_r-1$. Thus,
\begin{align*}
\rho(G') x_1^{\prime}&=(\kappa^h_r-1)x_{v_1}+p_1x_1+(n_1-p_1-1)x_1^{\prime},\\
\rho(G') x_1&=(\kappa^h_r-1)x_{v_1}+(p_1-1)x_1+(n_1-p_1)x_1^{\prime}+x_u,\\
\rho(G') x_i&=(\kappa^h_r-1)x_{v_1}+(p_i-1)x_i+(n_i-p_i)x_i^{\prime}+x_u,
\end{align*}
where $2\leq i\leq r$. Hence,
\begin{equation}
\begin{split}\label{equ::13}
p_1x_1+(n_1-p_1)x_1^{\prime}&> p_ix_i+(n_i-p_i-1)x_i^{\prime}\\ \big(\rho(G')+1\big)x_1&=\big(\rho(G')+1\big)x_1^{\prime}+x_u.
\end{split}
\end{equation}

We assert that $n_1\geq n_i$ for $2\leq i\leq r$. If not, there exists some $k$ $(2\leq k\leq r)$ such that $n_1<n_k$. Obviously, $p_1>p_k$. Otherwise, we have $x_1<x_k$, a contradiction. Take a vertex $w\in V(C_k)$ such that $wu\notin E(G)$.
Let $G^*=G'+\{wv\mid v\in V(C_1)\}-\{wv\mid v\in V(C_k)\}$. Then
\begin{align*}
\rho(G_1)-\rho(G')&\geq x^T(A(G^*)-A(G'))x \\
&=2x_w(p_1x_1+(n_1-p_1)x_1^{\prime}- (p_kx_k+(n_k-p_k-1)x_k^{\prime}))\\
&>0
\end{align*}
due to (\ref{equ::13}).
It follows that $\rho(G^*)>\rho(G')$,
a contradiction. Thus, $n_1\geq n_i$ for $2\leq i\leq r$.

In this situation, suppose that $E_1=\{uu_1^i\mid p_1+1\leq i\leq \delta-\kappa^h_r+1\}$ and $E_2=\{uv\in G'\mid v\in \bigcup_{i=2}^r C_i\}$.
Let $G_2=G'+E_1-E_2$ and $y$ be the Perron vector of $A(G_2)$. Recall that $n_1\geq n_i$ for $2\leq i\leq r$. Then
\begin{align*}
\rho(G_2)y_1&=(\kappa^h_r-1)y_{v_1}+\big(\delta-\kappa^h_r\big)y_1+(n_1-\delta+\kappa^h_r-1)y_1^{\prime}+y_u,\\
\rho(G_2)y_1'&=(\kappa^h_r-1)y_{v_1}+\big(\delta-\kappa^h_r+1\big)y_1+\big(n_1-\delta\kappa^h_r \big)y_1',\\
\rho(G_2)y_i'&=(\kappa^h_r-1)y_{v_1}+\big(n_i-1 \big)y_i',
\end{align*}
where $2\leq i\leq r$, which indicates that
\begin{align*}
\big(\rho(G_2)+1\big)y_1=\big(\rho(G_2)+1\big)y_1^{\prime}+y_u\text{ and } y_1'>y_i'.
\end{align*}
Combining this with (\ref{equ::13}), for any $2\leq i, j\leq r$, it follows that
\begin{align*}
&~~~x_{u}y_1+x_1^{\prime}y_{u}-x_{u}y_i^{\prime}-x_jy_{u}=x_u(y_1-y_i^{\prime})+y_u(x_1^{\prime}-x_j)\\
&>x_u(y_1-y_1^{\prime})+y_u(x_1^{\prime}-x_1)=\frac{x_uy_u}{\rho(G_2)+1}-\frac{x_uy_u}{\rho(G')+1}\\
&=\frac{\big(\rho(G')-\rho(G_2)\big)x_uy_u}{(\rho(G_2)+1)(\rho(G')+1)}.
\end{align*}
If $\rho(G')-\rho(G_2)<0$, then $\rho(G')<\rho(G_2)$, and if $\rho(G')-\rho(G_2)\geq0$, then we may assume that
$x_l=\max\{x_i\mid 2\leq i\leq r\}$ and $y_k'=\max\{y_i'\mid 2\leq i\leq r\} $, and then
\begin{align*}
y^T(\rho(G_2)-\rho(G'))x& =y^T(A(G_2)-A(G'))x \\
&=\sum_{uu_1^i\in E_1}(x_uy_{u_1^i}+x_{u_1^i}y_{u})-\sum_{uv\in E_2}(x_{u}y_{v}+x_{v}y_{u})\\
&\geq(\delta-\kappa^h_r-p_1+1)(x_{u}y_1+x_1^{\prime}y_{u}-x_{u}y_k^{\prime}-x_ly_{u})\\
&>(\delta-\kappa^h_r-p_1+1)\frac{\big(\rho(G')-\rho(G_2)\big)x_uy_u}{(\rho(G_2)+1)(\rho(G')+1)}\\
&\geq0 ~~ (\text{since}~~\rho(G')\geq \rho(G_2)),
\end{align*}
a contradiction. Thus, we conclude that $\rho(G_2)>\rho(G')$. In this case, $N_{G_2}(u)\cap V(C_1)=\delta-\kappa^h_r+1$.
Furthermore, take a vertex $z\in V(C_1)\cap N(u)$. Let $G_3=G_2+\{zv\mid v\in \bigcup_{i=2}^rV(C_i)\}$ and $S^{\prime}=S-u+z$.
Then $|S^{\prime}|=\kappa^h_r$ and $G_3\in \mathcal{G}_{n,\delta}^{\kappa_r^h}$. It follows that $\rho(G_3)>\rho(G_2)$. Combining this with $\rho(G_2)>\rho(G')$, we have $\rho(G_3)>\rho(G^{\prime})$, a contradiction. Thus, $u\notin S$.

Suppose that $u\in V(C_1)$ by Claim 4.  Then $d_S(u)=
d_{G_1}(u)-d_{C_1}(u)\leq \kappa^h_r$.
{\flushleft\bf Claim 5.} $d_S(u)=\kappa_r^h$.

Otherwise, let $d_S(u)\leq \kappa_r^h-1$. Then $d_{C_1}(u)\geq \delta-\kappa^h_r\geq2$. Similar to Claim 2,  we get
that $x_{v_{t+1}}>x_1$. Take $G^*=G^{\prime}-\{u u_1^i \mid 1 \leq i \leq \kappa_r^h-t\}+\{u v_j \mid t+1 \leq j \leq
\kappa_r^h\}$. Then $G^*\in \mathcal{G}_{n,\delta}^{\kappa_r^h}$ and $\rho(G^*)>\rho(G')$ by Lemma \ref{lem::2.1}, a contradiction.

Suppose that $n_1\geq n_2\geq \cdots\geq
n_r$. By using the same analysis as Claim 3,  we have $u\in V(C_i)$ for $n_i=n_r$, where $1\leq i\leq r$. In what follows, we shall prove that $G^{\prime} \cong G_n^{\kappa^h_r, \delta}$ for $\kappa^h_r+1\leq\delta<\kappa^h_r+h$. In fact, by
the maximality of $\rho(G^{\prime})$, Claims 4, 5 and $u\in V(C_i)$ for $n_i=n_r$, we can deduce that $G^{\prime}-u \cong
K_{\kappa^h_r} \vee(K_{n_1} \cup K_{n_2} \cup \cdots \cup K_{n_{r}-1})$ for some integers $n_1 \geq n_2 \geq \cdots \geq
n_r$, where $\sum_{i=1}^{r} n_i=n-\kappa^h_r$, $u\in V(C_r)$ and $d_S(u)=\kappa^h_r$. If $n_i=h+1$ for all $2\leq i\leq r$, then the result holds. Similar to Case 1, we get $n_i=h+1$ for all $2\leq i\leq r$. It follows that $G' \cong G_n^{\kappa_r^h, \delta}$, as desired.
{\flushleft\bf{Case 3.}} $\delta\geq \kappa^h_r+h$.

Recall that $|V(C_i)|=n_i$ for $1 \leq i \leq r$. It is easy to obtain that
$$
\rho(G) \leq \rho(K_{\kappa_r^h} \vee(K_{n_1} \cup K_{n_2} \cup \cdots \cup K_{n_r})),
$$
with equality if and only if $G \cong K_{\kappa_r^h} \vee(K_{n_1} \cup K_{n_2} \cup \cdots \cup K_{n_r})$.
We assert that $n_i\geq \delta+1-\kappa_r^h$ for $2\leq i\leq r$. If not, $n_i< \delta-\kappa_r^h+1$. Then there exists $u\in
V(C_i)$ such that $d_{G^{\prime}}(u)<\delta$, a contradiction. On the other hand, we assert that $n_i\leq \delta+1-\kappa_r^h$ for $2\leq
i\leq r$. Otherwise, there exists $n_k$ such that $n_k> \delta+1-\kappa_r^h$ where $2\leq k\leq r$. Let $x_v=x_1$ for $v\in V(K_{n_1})$ and $x_v=x_k$ for $v \in V(K_{n_k})$. Recall that $x$ is the Perron vector of $A(G')$. Since
$$(\rho(G^{\prime})-n_i+1)x_i=\kappa^h_rx_{v_1}$$
for $1\leq i\leq r$,
$x_1\geq x_k$ due to $n_1 \geq n_k$. Take $w \in V(K_{n_k})$ and $G_1=G^{\prime}-\{w v \mid v \in V(K_{n_k}) \backslash\{w\}\}+\{w v \mid v \in V(K_{n_1})\}$.
Then $G_1\in \mathcal{G}_{n,\delta}^{\kappa_r^h}$ and $\rho(G_1)>\rho(G')$ by Lemma \ref{lem::2.1}, a contradiction. Therefore, $n_i=\delta+1-\kappa_r^h$ for $2\leq i\leq r$. By
the maximality of $\rho(G^{\prime})$, we conclude that $G^{\prime} \cong K_{\kappa_r^h}
\vee(K_{n-\kappa_r^h-(\delta+1-\kappa_r^h)(r-1)} \cup (r-1) K_{\delta+1-\kappa_r^h})$.

This completes the proof.
\end{proof}

\section{ Proof of Theorem \ref{thm::1.3}}

Our one primary tool is sharp upper bound on the spectral radius, which was obtained by Hong, Shu, and Fang \cite{Y.H} and
Nikiforov \cite{V.N}, independently.

\begin{lem}[See \cite{Y.H} and \cite{V.N}]\label{lem::3.1}
Let $G$ be a connected graph on $n$ vertices and $m$ edges with minimum degree $\delta\geq 1$. Then $$\rho(G)\leq
\frac{\delta-1}{2}+\sqrt{2m-n\delta+\frac{(\delta+1)^2}{4}},$$ with equality if and only if $G$ is a $\delta$-regular graph
or a bidegreed graph in which each vertex is of degree either $\delta$ or $n-1$.
\end{lem}

\begin{lem}\label{lem::3.2}
Let a, b and t be three integers.  If $b>a>t\geq1$ and $|a-b|<1$, then $$ab>(a+t)(b-t).$$
\end{lem}
\renewcommand\proofname{\bf Proof}
\begin{proof}
Since $b>a>t\geq1$ and $|a-b|<1$, we have $$ab-(a+t)(b-t)=at-bt+t^2>t^2-t\geq0,$$
as required.
\end{proof}
Denote by $\mathcal{K}_{n,\delta}^{\lambda_r^h}$ the set of graphs with $h$-extra $r$-component edge-connectivity $\lambda^h_r$ and with minimum degree $\delta$
obtained from $K_{n-(r-1)(h+1)}\cup (r-2)K_{h+1}\cup K_h\cup K_1$ by adding $t$ edges between $K_1$ and
$K_h$ for $1\leq t\leq \delta$,  and $r-1$ edges between a vertex in $K_{n-(r-1)(h+1)}$ and a vertex in each of
$K_{h+1}$ and $K_h$, and then adding $\lambda_r^h-r+1-\delta+t$ edges between $K_{n-(r-1)(h+1)}$ and $(r-2)K_{h+1}\cup
K_h\cup K_1$,
where $r\geq2$ and $\lambda_r^h\geq r-1$. Clearly, $\mathcal{K}_{n,\delta}^{\lambda_r^h} \subset \mathcal{B}_{n,\delta}^{\lambda_r^h}$.

\begin{lem}\label{lem::3.3}
Suppose that $r$, $\delta$, $n$ and $h$ are four positive integers such that $n\geq (\lambda_r^h+1)(h+1)^2$,
$\lambda_r^h\geq r-1$ and $h\geq\delta$.  Let $G\in \mathcal{K}_{n,\delta}^{\lambda_r^h}$. Then $n-(r-1)(h+1)-1<\rho(G)<
n-(r-1)(h+1)$.

\renewcommand\proofname{\bf Proof}
\begin{proof}
For $1\leq\delta\leq h$, $\rho(G)>\rho(K_{n-(r-1)(h+1)})=n-(r-1)(h+1)-1$ due to $G$ contains $K_{n-(r-1)(h+1)}$ as a proper subgraph and $n\geq (\lambda_r^h+1)(h+1)^2$. On the other hand,
\begin{align*}
e(G)&\leq\binom{n-(r-1)(h+1)}{2}+(r-2)\binom{h+1}{2}+\binom{h}{2}+\lambda_r^h+\delta\\
&\leq\binom{n-(r-1)(h+1)}{2}+(r-1)\binom{h+1}{2}+\lambda_r^h\quad(\text {since }\delta\leq h)\\
&=\frac{(n-(r-1)(h+1))(n-(r-1)(h+1)-1)}{2}+\frac{h(h+1)(r-1)}{2}+\lambda_r^h.
\end{align*}
Notice that $\lambda_r^h\geq r-1$, $r\geq 2$ and $h\geq \delta$. Thus,
\begin{align*}
&~~~\Big(n-(r-1)(h+1)-\frac{\delta-1}{2}\Big)^2-2e(G)+n\delta-\frac{(\delta+1)^2}{4}\\
&=\Big(n-(r-1)(h+1)-\frac{\delta-1}{2}\Big)^2-(n-(r-1)(h+1))(n-(r-1)(h+1)-1)-\\
&~~~h(h+1)(r-1)-2\lambda_r^h+n\delta-\frac{(\delta+1)^2}{4}\\
&=\big((r-1)(h+1)-1\big)\delta+2n-2(r-1)(h+1)-h(h+1)(r-1)-2\lambda_r^h\\
&\geq (r-1)(h+1)-1+2(\lambda_r^h+1)(h+1)^2-2(r-1)(h+1)-h(h+1)(r-1)-2\lambda_r^h\\
& \quad(\text {since }\delta\geq 1 \text{ and } n\geq (\lambda_r^h+1)(h+1)^2)\\
&=2\lambda_r^h h^2-rh^2+3h^2+4\lambda_r^h h-2rh+6h-r+2\\
&\geq 2(r-1)h^2-rh^2+3h^2+4(r-1) h-2rh+6h-r+2\quad(\text {since }\lambda_r^h\geq r-1)\\
&=rh^2+h^2+2rh+2h-r+2\\
&>0\quad(\text {since } r\geq2 \text{ and } h>\delta\geq1).
\end{align*}
Combining this with Lemma \ref{lem::3.1},
\begin{align*}
\rho(G)& \leq \frac{\delta-1}{2}+\sqrt{2e(G)-n\delta+\frac{(\delta+1)^2}{4}}< n-(r-1)(h+1).
\end{align*}
This completes the proof.
\end{proof}
\end{lem}

Recall that $B^{\lambda^h_r,\delta}_n$ is a graph obtained from $(K_{n-(r-1)(h+1)-\lambda_r^h+r-2}\vee
K_{\lambda_r^h-r+2})\cup(r-2)K_{h+1}\cup (K_\delta\vee K_{h-\delta})\cup K_1$ by adding $\delta$ edges between $K_1$ and $K_\delta$, and $\lambda_r^h-r+2$ edges between some vertex in a copy of $K_{h+1}$ and all vertices in $K_{\lambda_r^h-r+2}$, and finally adding $r-2$ edges between a vertex in $K_{\lambda_r^h-r+2}$
and a vertex in each remaining $(r-3)$ copy of $K_{h+1}$ and $K_\delta$, respectively  For $X,Y\subset V(G)$, we denote by $E(X)$ the set of edges in $G[X]$,
and $e(X)=|E(X)|$. Let $E(X,Y)$ be the set of edges with one endpoint in $X$ and one endpoint in $Y$,  and let
$e(X,Y)=|E(X,Y)|$.
\begin{lem}\label{lem::3.4}
Suppose that $r$, $\delta$, $n$ and $h$ are four positive integers such that $n\geq (\lambda_r^h+1)(h+1)^2$,
$\lambda_r^h\geq r\geq 2$ and $h\geq\delta$. Let $G\in \mathcal{K}_{n,\delta}^{\lambda_r^h}$. Then
$$\rho(G)\leq\rho(B_n^{\lambda_r^h,\delta}),$$ with equality if and only if $G\cong B_n^{\lambda_r^h,\delta}$.
\end{lem}
\begin{proof}
Suppose that $G'\in\mathcal{K}^{\lambda_r^h}_{n,\delta}$ is a graph that attains the maximum spectral radius, where
$n\geq(\lambda_r^h+1)(h+1)^2$ and $\lambda_r^h\geq r\geq 2$. Then
\begin{equation}\label{equ::14}
\begin{aligned}
\rho(G)\leq \rho(G').
\end{aligned}
\end{equation}
For the convenience of writing, we will write $n_1=n-(r-1)(h+1)$. Let $V(G')=V(C_1)\cup V(C_2)\cup\cdots\cup V(C_r)$ such
that $V(C_1)=V(K_{n_1})=\{v_1,v_2,\ldots,v_{n_1}\}$, $V(C_i)=V(K_{h+1})=\{u_i^1,u_i^2,\ldots,u_i^{h+1}\}$ for $2\leq
i\leq r-1$ and $V(C_r)=V(K_{h}\cup K_1)=\{u_r^1,u_r^2,\ldots,u_r^{h+1}\}$. Suppose that $x$ is the Perron vector of $A(G')$ and
$\rho(G') =\rho'$. Without loss of generality, assume that $x_{v_{l+1}}\leq x_{v_{l}}$, $x_{u^1_{i+1}}\leq x_{u^1_{i}}$
and $x_{u_i^{j+1}}\leq x_{u_i^{j}}$ for $1\leq l\leq n_1-1$, $2\leq i\leq r$ and $1\leq j\leq h$. Since $G'$ is a connected
graph and $e(C_i,C_j)=0$ for $2\leq i<j\leq r$, it follows that $|N_{G'}(C_i)|=|N_{C_1}(C_i)|\geq 1$ for $2\leq i\leq r$.
Suppose that $d_{G'}(u_r^q)=\delta$, where $1\leq q\leq h+1$.
We first claim that $q\neq 1$. If not, let $q=1$, $e(u_r^{1},C_1)=a_1$ and $e(u_r^{h+1},C_1)=a_2$. Then $a_1\geq
a_2$ due to $x_{u_r^{1}}\geq x_{u_r^{h+1}}$ and $N_{C_r}(u_r^1)\subset N_{C_r}(u_r^{h+1})$. Suppose that
$$G_1=G'+\{u_r^1u_r^j \mid \delta-a_1+1\leq j\leq h+1\}-\{u_r^{h+1}u_r^j\mid \delta-a_2+1\leq j\leq h\}.$$ Then
$\rho(G_1)>\rho(G')$ by Lemma \ref{lem::2.1}, a contradiction.
We next assert that $v_1\in N_{C_1}(C_i)$ for $2\leq i\leq r$.
Otherwise, there exists some $C_j$ such that $v_1\notin N_{C_1}(C_j)$, where $2\leq j\leq r$. Let $v_t\in N_{C_1}(C_j)$,
where $2\leq t\leq n_1$. Then we assume that $G_2=G'-\{v_tv\mid v\in V(C_j)\}+\{v_1v\mid v\in V(C_j)\}.$ Hence, $G_2\in
\mathcal{K}_{n,\delta}^{\lambda_r^h}$ and $\rho(G_2)>\rho(G')$ by Lemma \ref{lem::2.1}, a contradiction.
This implies that $v_1\in N_{C_1}(C_i)$ for all $2\leq i\leq r$.  In the following, we claim that, except for the vertex
$u_r^q$, $N_{G'}(v_{l+1})\subseteq N_{G'}[v_{l}]$ and $N_{G'}(u_i^{j+1})\subseteq N_{G'}[u_i^{j}]$ for $1\leq l\leq n_1-1$,
$2\leq i\leq r$ and $1\leq j\leq h$. Otherwise, suppose that there exists $i$, $j$ with $i<j$ such that
$N_{G'}(v_{j})\nsubseteq N_{G'}[v_i]$. Let $w\in N_{G'}(v_{j})\backslash N_{G'}[v_i]$ and $G_3=G'-wv_j+wv_i$. It is clear that
$G_3\in \mathcal{K}_{n,\delta}^{\lambda_r^h}$.
Note that $x_{v_{j}}\leq x_{v_{i}}$. Then $\rho(G_3)>\rho(G')$ by Lemma \ref{lem::2.1}, a contradiction.
This implies that $N_{G'}(v_{l+1})\subseteq N_{G'}[v_{l}]$ for $1\leq l\leq n_1-1$. Similarly,
$N_{G'}(u_i^{j+1})\subseteq N_{G'}[u_i^{j}]$ for $2\leq i\leq r$ and $1\leq j\leq h$. Obviously, $v_1\in N_{C_1}(u_i^1)$ for
$2\leq i\leq r$.
Let $d_{C_1}(u_i^1)=t_i$ and $d_{C_i}(v_1)=b_i$ for $2\leq i\leq r$. Recall that $d_{C_1}(u_i^1)\geq 1$ for $1\leq i\leq r$,
which means $u_i^1v_1\in E(G')$. Thus, $b_i\geq 1$. If $t_2=\lambda_r^h-r+2$, then $G'\cong B^{\lambda_r^h,\delta}_{n}$.
Furthermore, by (\ref{equ::14}), we deduce that
$$\rho(G)\leq \rho(B^{\lambda_r^h,\delta}_{n}),$$
with equality if and only if $G\cong B^{\lambda_r^h,\delta}_{n}$, as required. Next, we consider the case of $t_2\leq
\lambda_r^h-r+1$. Let $e(C_1,C_i)=p_i$ and $d_{C_1}(u_i^2)=s_i$ for $2\leq i\leq r$. Then we assert that $t_i=p_i$. If not,
there exists $l$ such that $t_l<p_l$ where $2\leq l\leq r$. Furthermore,  $N_{G'}[v_{n_1}]=V(C_1)$ due to
$n_1=n-(r-1)(h+1)\geq (\lambda_r^h+1)(h+1)^2-(r-1)(h+1)>\lambda_r^h$. Then
\begin{align*}
\rho' x_{v_{n_1}}&=\sum_{i=1}^{n_1-1}x_{v_i},\\
\rho'
x_{v_1}&\leq\sum_{i=2}^{n_1}x_{v_i}+\sum_{i=2}^{r}b_ix_{u_i^1}\leq\sum_{i=2}^{n_1}x_{v_i}+b_lx_{u_l^1}+(\lambda_r^h-p_l)x_{u_2^1},\\
\rho' x_{v_{p_l}}&\geq\sum_{i=1}^{p_l-1}x_{v_i}+\sum_{i=p_l+1}^{n_1}x_{v_i},\\
\rho' x_{u_2^1}&=\sum_{i=1}^{t_2}x_{v_i}+\sum_{i=2}^{b_2}x_{u_2^i}+(h+1-b_2)x_{u_2^{h+1}},
\end{align*}
from which we obtain that $x_{v_1}\leq x_{v_{p_l}}+\frac{\lambda_r^h-p_l+b_l}{\rho'+1}x_{u_2^1}$ and
$x_{v_{n_1}}\geq\frac{\rho'-h}{\rho'-n_1+t_2+1}x_{u_2^1}$.

Suppose that $E_1=\{u_l^1v_i\mid t_l+1\leq i\leq p_l\}$ and $E_2=\{u_l^iv_j\mid 2\leq i\leq b_l, 1\leq j\leq s_l\}$, where
$2\leq l\leq r$.
Let $G''=G'-E_2+E_1$. Clearly, $G'' \in\mathcal{K}_{n,\delta}^{\lambda^h_r}$. Let $y$ be the Perron vector of $A(G'')$ and $\rho(G'')=\rho''$. By symmetry,
$y_{v_i}=y_{v_{p_2+1}}$ for $p_2+2\leq i\leq n_1$, and except for the vertex $u_r^q$,  $y_{u_i^j}=y_{u_i^2}$ for
$2\leq i\leq r$ and $3\leq j\leq h+1$. Then
$$
\left\{\begin{array} { l }
{ \rho'' y_{v_1}\leq\sum_{i=2}^{n_1}y_{v_i}+y_{u_l^1}+(\lambda_r^h-p_l)y_{u_2^1} , } \\
{ \rho''y_{v_{p_l}}\geq \sum_{i=1}^{p_l-1}y_{v_i}+\sum_{i=p_l+1}^{n_1}y_{v_i}+y_{u_l^1} , }\\
{\rho''y_{u_l^1}=\sum_{i=1}^{p_l}y_{v_i}+\sum_{i=2}^{h+1}y_{u_l^i} ,}

\end{array} \quad \left\{\begin{array}{l}
\rho''y_{u_l^2}=y_{u_l^1}+\sum_{i=3}^{h+1}y_{u_l^i},\\
\rho''y_{v_{n_1}}=\sum_{i=1}^{n_1-1}y_{v_i}, \\
\rho''y_{u_2^1}=\sum_{i=1}^{p_2}y_{v_i}+(h+1)y_{u_2^2} ,
\end{array}\right.\right.
$$
and hence, $y_{v_1}\leq y_{v_{p_l}}+\frac{\lambda_r^h-p_l}{\rho''+1}y_{u_2^1},$
$y_{u_l^1}\geq y_{u_l^2}+\frac{p_l}{\rho''+1}y_{v_{p_l}}$ and
$y_{v_{n_1}}\geq \frac{\rho''-h+1}{\rho''-n_1+p_2+1}y_{u_{2}^1}$.
Combining this with $x_{v_1}\leq x_{v_{p_l}}+\frac{\lambda_r^h-p_l+b_l}{\rho'+1}x_{u_2^1}$ and
$x_{v_{n_1}}\geq\frac{\rho'-h}{\rho'-n_1+t_2+1}x_{u_2^1}$, we have
\begin{align*}
&~~~y^T(\rho''-\rho')x=y^T(A(G'')-A(G'))x \\
&=\sum_{u_l^1v_i\in E_1}(x_{u_l^1}y_{v_i}+x_{v_i}y_{u_l^1})-\sum_{u_l^jv_i\in E_2}(x_{u_l^j}y_{v_i}+x_{v_i}y_{u_l^j})\\
&\geq(p_l-t_l)(x_{u_l^1}y_{v_{p_l}}+x_{v_{p_l}}y_{u_l^1}-x_{u_l^1}y_{v_1}-x_{v_1}y_{u_l^2})\\
&\geq (p_l-t_l)\Big(x_{u_l^1}(y_{v_{p_l}}-y_{v_1})+x_{v_{p_l}}y_{u_l^1}-x_{v_1}y_{u_l^2}\Big)\\
&\geq(p_l-t_l)\Big(x_{u_l^1}(y_{v_{p_l}}-y_{v_{p_l}}-\frac{\lambda_r^h-p_l}{\rho''+1}y_{u_2^1})+x_{v_{p_l}}(y_{u_l^2}+\frac{p_l}{\rho''+1}y_{v_{p_l}})-(x_{v_{p_l}}+\frac{\lambda_r^h-p_l+b_l}{\rho'+1}x_{u_2^1})y_{u_l^2}\Big)\\
&=(p_l-t_l)\Big(\frac{p_l}{\rho''+1}x_{v_{p_l}}y_{v_{p_l}}-\frac{\lambda_r^h-p_l+b_l}{\rho'+1}x_{u_2^1}y_{u_l^2}-\frac{\lambda_r^h-p_l}{\rho''+1}x_{u_l^1}y_{u_2^1}\Big)\\
&\geq(p_l-t_l)\Big(\frac{p_l}{\rho''+1}x_{v_{n_1}}y_{v_{n_1}}-\frac{\lambda_r^h-p_l+b_l}{\rho'+1}x_{u_2^1}y_{u_l^2}-\frac{\lambda_r^h-p_l}{\rho''+1}x_{u_l^1}y_{u_2^1}\Big)\\
&\geq (p_l-t_l)\Big(\frac{p_l(\rho'-h)(\rho''-h+1)}{(\rho''\!+\!1)(\rho'\!-\!n_1\!+\!t_2\!+\!1)(\rho''\!-\!n_1\!+\!p_2\!+\!1)}x_{u_2^1}y_{u_2^1}\!-\!\big(\frac{\lambda_r^h\!-\!p_l\!+\!b_l}{\rho'+1}+\frac{\lambda_r^h-p_l}{\rho''+1}\big)x_{u_2^1}y_{u_2^1}\Big)\\
&\geq
(p_l-t_l)\Big(\frac{p_l(\rho'-h)(\rho''-h+1)}{(\rho''+1)(t_2+1)(p_2+1)}-\frac{(\lambda_r^h-p_l)(\rho'+1)+(\lambda_r^h-p_l+b_l)(\rho''+1)}{(\rho''+1)(\rho'+1)}\Big)x_{u_2^1}y_{u_2^1}\\
&\geq(p_l-t_l)\Big(\frac{p_l(\rho'-h)(\rho''-h+1)-2\lambda_r^h(\rho'+2)}{(\rho''+1)(t_2+1)(p_2+1)}\Big)x_{u_2^1}y_{u_2^1}\\
&\geq(p_l-t_l)\Big(\frac{2(\rho'+2)(\rho''-2h-1)-2\lambda_r^h(\rho'+2)}{(\rho''+1)(t_2+1)(p_2+1)}\Big)x_{u_2^1}y_{u_2^1}\quad(\text{by
Lemma \ref{lem::3.2}} \text{ and }~ p_l>t_l\geq1)\\
&\geq(p_l-t_l)\Big(\frac{2(\rho'+2)(\rho''-2h-\lambda_r^h-1)}{(\rho''+1)(t_2+1)(p_2+1)}\Big)x_{u_2^1}y_{u_2^1}\\
&>0\quad(\text{since } \rho''>n-(r-1)(h+1)-1,~n>(\lambda_r^h+1)(h+1)^2\text{ and }~  p_l>t_l).
\end{align*}
It follows that
\begin{align}
\rho''>\rho', \label{equ::15}
\end{align}
a contradiction. This implies that $d_{C_1}(u_i^1)=t_i=p_i\geq 1$,
$\sum_{i=2}^{r}=t_i=\lambda_r^h$ and $d_{C_1}(u_i^j)=0$, where $2\leq i\leq r$ and $2\leq j\leq h+1$.

In this situation, if $p_2=\lambda_r^h-r+2$, then $G'\cong B^{\lambda_r^h,\delta}_{n}$. Combining this with (\ref{equ::14}), we can
deduce that
$$\rho(G)\leq \rho(B^{\lambda_r^h,\delta}_{n}),$$
with equality if and only if $G\cong B^{\lambda_r^h,\delta}_{n}$, as required. In the following, we consider the case of
$p_2\leq \lambda_r^h-r+1$. Note that $y_{v_i}=y_{v_{r+1}}$ for $p_2+2\leq i\leq n_1$. Then
\begin{align*}
\rho''
y_{v_{p_2+1}}&=\sum_{i=1}^{p_2}y_{v_i}+\sum_{i=p_2+2}^{n_1}y_{v_{p_2+1}}=\sum_{i=1}^{p_2}y_{v_i}+(n_1-p_2-1)y_{v_{p_2+1}},\\
\rho'' y_{v_2}&\geq y_{v_1}+\sum_{i=3}^{p_2}y_{v_i}+(n_1-p_2)y_{v_{p_2+1}}+(r-1)y_{u_2^1},\\
\rho''y_{u_2^1}&=\sum_{i=1}^{p_2}y_{v_i}+hy_{u_2^2},
\end{align*}
and hence,
\begin{align}
y_{v_2} &\geq y_{v_{p_2+1}}+\frac{r-1}{\rho''+1}y_{u_2^1},\label{equ::16}\\
y_{v_{p_2+1}} &>\frac{\rho''-h}{\rho''-n_1+p_2+1}y_{u_2^1}.\label{equ::17}
\end{align}
Suppose that $E_1=\{u_2^1v_i\mid p_2+1\leq i\leq \lambda_r^{h}-r+2\}$ and $E_2=\{u_i^1v_j\mid 3\leq i\leq r, 2\leq j\leq
p_i\}$. Let $G'''=G''+E_1-E_2$. Clearly, $G'''\cong B^{\lambda_r^h,\delta}_{n}$. Let $z$ be the Perron vector of $A(G''')$ and
$\rho(G''')=\rho'''$. Then $z_{u_i^1}>z_{u_i^2}$ for $2\leq i\leq r$, $z_{v_j}=z_{v_2}$ for $3\leq j\leq
n_1$, and except for the vertex $u_r^q$, $z_{u_i^l}=z_{u_i^2}$ for $2\leq i\leq r$ and $3\leq l\leq h+1$, and hence
\begin{align}
\rho'''z_{u^1_3}&\geq z_{v_1}+hz_{u^2_3},\label{equ::18}\\
\rho'''z_{u^1_2}&=\sum_{i=1}^{\lambda_r^h-r+2}z_{v_i}+hz_{u^2_2},\label{equ::19}\\
\rho'''z_{v_2}&=z_{v_1}+\sum_{i=3}^{n_1}z_{v_i}=z_{v_1}+(n_1-2)z_{v_2}.\label{equ::20}
\end{align}
From (\ref{equ::18})-(\ref{equ::20}), we get $z_{u_2^1}\geq\frac{2}{\rho'''}z_{v_2}+z_{u_3^1}$ and
 $z_{v_2}>\frac{\rho'''-h}{\rho'''-n_1+2}z_{u_3^1}.$
Note that $y_{u_1^1}\geq y_{u_i^1}$ for $2\leq i\leq r$, $y_{v_1}\geq y_{v_j}$ for $2\leq j\leq t$. Recall that
$\rho'''-\rho''>-1$ and $\rho''<n_1$ by Lemma \ref{lem::3.3}. Combining this with (\ref{equ::16}) and (\ref{equ::17}), we have
\begin{align*}
&~~~z^T(\rho'''-\rho'')y=z^T(A(G''')-A(G''))y \\
&=\sum_{u_2^1v_i\in E_1}(y_{u_2^1}z_{v_i}+y_{v_i}z_{u_2^1})-\sum_{u_i^1v_j\in E_2}(y_{u_i^1}z_{v_j}+y_{v_j}z_{u_i^1})\\
&\geq(\lambda_r^h-r-p_2+2)(y_{u_2^1}z_{v_2}+y_{v_{p_2+1}}z_{u_2^1}-y_{u_2^1}z_{v_2}-y_{v_2}z_{u_3^1})\\
&= (\lambda_r^h-r-p_2+2)(y_{v_{p_2+1}}z_{u_2^1}-y_{v_2}z_{u_3^1})\\
&\geq(\lambda_r^h-r-p_2+2)\Big(y_{v_{p_2+1}}z_{u_3^1}+\frac{2(\rho''-h)y_{u_2^1}z_{v_2}}{(\rho''-n_1+p_2+1)\rho'''}-y_{v_{p_2+1}}z_{u_3^1}-\frac{(r-1)y_{u_2^1}z_{u_3^1}}{\rho''+1}\Big)\\
&\geq\frac{\lambda_r^h-r-p_2+2}{\rho''+1}\Big(\frac{2(\rho''-h)y_{u_2^1}z_{v_2}}{p_2+1}-(r-1)y_{u_2^1}z_{u_3^1}\Big)
\quad(\text {since }\rho''+1>\rho'''\text{ and }\rho''<n_1)\\
&\geq\frac{\lambda_r^h-r-p_2+2}{\rho''+1}\Big(\frac{2(\rho'''-h)(\rho''-h)}{(p_2+1)(\rho'''-n_1+2)}-(r-1)\Big)y_{u_2^1}z_{u_3^1}\\
&\geq\frac{\lambda_r^h-r-p_2+2}{\rho''+1}\Big(\frac{(\rho'''-h)(\rho''-h)}{p_2+1}-(r-1)\Big)y_{u_2^1}z_{u_3^1}\quad(\text
{since }\rho''-n_1+2<2)\\
&>0\quad(\text {since }\rho'',~\rho'''>n-(r-1)(h+1)-1\text{ and }p_2\leq \lambda_r^h-r+1).
\end{align*}
It follows that $\rho'''>\rho''$. Combining this with (\ref{equ::15}), we have $\rho'''>\rho'$, a contradiction. Thus,
$$\rho(G)\leq \rho(B^{\lambda_r^h,\delta}_{n}),$$
 with equality if and only if $G\cong B^{\lambda_r^h,\delta}_{n}$.  This completes the proof.
\end{proof}

\begin{lem}\label{lem::3.5}
Let $a$ and $b$ be two positive integers. If $a\geq b\geq 3$, then
$$\binom{a}{2}+\binom{b}{2}<\binom{a+1}{2}+\binom{b-1}{2}.$$
\end{lem}
\renewcommand\proofname{\bf Proof}
\begin{proof}
Since $a\geq b\geq 3$, we have $$\binom{a+1}{2}+\binom{b-1}{2}-\binom{a}{2}-\binom{b}{2}=a-b+1>0,$$
as required.
\end{proof}

\renewcommand\proofname{\bf Proof of Theorem \ref{thm::1.3}}
\begin{proof}
Suppose that $G'\in \mathcal{B}_{n,\delta}^{\lambda_r^h}$ is a graph that attains the maximum spectral radius. Thus,
$$\rho(G)\leq\rho(G').$$
Since $n\geq (\lambda_r^h+1)(h+1)^2$,
$r\geq2$ and $\delta\leq h$, there exists some $F \subseteq E(G')$ with
$\lambda_r^h=|F|$ such that $G\!-\!F$ contains $p$ components $C_1, C_2, \ldots, C_p~(p \geq r)$. We assert that $p=r$. Otherwise, $p>r$ and $F'=F\!-\!E(G'\!-\!C_p,C_p)$. Obviously,
$\lambda_r^h=|F'|<|F|$, a contradiction. Let $|V(C_i)|=n_i$ for $1 \leq i \leq r$. Without loss of generality, assume that
$n_1 \geq n_2 \geq$ $\cdots \geq n_r$.
Note that $B_n^{\lambda_r^h,\delta}\in \mathcal{B}_{n,\delta}^{\lambda_r^h}$ and $B_n^{\lambda_r^h,\delta}$ contains
$K_{n-(r-1)(h+1)}$ as a proper subgraph. Thus,
\begin{align}
\rho(G') \geq \rho\big(B_n^{\lambda_r^h,\delta}\big)>\rho\big(K_{n-(r-1)(h+1)}\big)=n-(r-1)(h+1)-1.\label{equ::21}
\end{align}
We first have the following Claim.

{\flushleft\bf Claim 1.} $n_i=h+1$ for $2 \leq i \leq r$.

Otherwise, there exists some $j$ $(2 \leq j \leq r)$ such that $n_j \geq h+2$.
Notice that $1\leq\delta\leq h$. By Lemma \ref{lem::3.5},
\begin{align*}
e(G')&=\sum_{i=1}^r e(C_i)+\sum_{1\leq i<j\leq r}e(C_i,C_j)=\sum_{i=1}^r e(C_i)+\lambda_r^h\\
&\leq\binom{n-(r-1)(h+1)-1}{2}+\binom{h+2}{2}+(r-3)\binom{h+1}{2}+\binom{h}{2}+\delta+\lambda_r^h\\
&=\frac{\big(n-(r-1)(h+1)-1\big)\big(n-(r-1)(h+1)-2\big)}{2}+\frac{(h+1)(h+2)}{2}+\\
&~~~~\frac{(r-3)(h+1)h}{2}+\frac{h(h+1)}{2}+\delta+\lambda_r^h.\\
\end{align*}
Note that $G'$ is connected. Then $\lambda_r^h\geq r-1$. Combining $n \geq (\lambda_r^h+1)(h+1)^2$, $r\geq 2$ and $h \geq
\delta\geq1$,
we get
\begin{align*}
&~~\Big((n-(r-1)(h+1)-1)-\frac{\delta-1}{2}\Big)^2-2e(G')+n\delta-\frac{(\delta+1)^2}{4}\\
&=\Big((n-(r-1)(h+1)-1)-\frac{\delta-1}{2}\Big)^2-(h+1)(h+2)-h(h+1)(r-3)-2\lambda_r^h-\\
&~~(n-(r-1)(h+1)-1)(n-(r-1)(h+1)-2)-h(h+1)-2\delta+n\delta-\frac{(\delta+1)^2}{4}\\
&=\big((r-1)(h+1)-3\big)\delta-rh^2+h^2-3rh-2\lambda_r^h+3h+2n-2r-2\\
&\geq (r-1)(h+1)-3-rh^2+h^2-3rh-2\lambda_r^h+3h+2(\lambda_r^h+1)(h+1)^2-2r-2\\
&\quad(\text {since } h\geq\delta \geq 1, r\geq2 \text { and } n \geq (\lambda_r^h+1)(h+1)^2).\\
&=2\lambda_r^hh^2-rh^2+4\lambda_r^hh+3h^2-2rh+6h-r-3\\
&\geq 2(r-1)h^2-rh^2+4(r-1)h+3h^2-2rh+6h-r-3\quad(\text {since } \lambda_r^h\geq r-1)\\
&=rh^2+h^2+2rh+2h-r-3\\
&\geq0\quad(\text {since } h\geq \delta\geq1 \text { and } r\geq 2).
\end{align*}
Furthermore, by Lemma \ref{lem::3.1}, we obtain that
\begin{align*}
\rho(G')& \leq \frac{\delta-1}{2}+\sqrt{2e(G')-n\delta+\frac{(\delta+1)^2}{4}}< n-(r-1)(h+1)-1,
\end{align*}
which contradicts (\ref{equ::21}). So,  $n_i=h+1$ for $2 \leq i \leq r$, and hence $n_1=n-(r-1)(h+1)$, as desired.

Choose a vertex $u\in V(G')$ such that $d_{G'}(u)=\delta$. Due to the uncertainty of the minimum degree vertex $u$, we will now present the following  Claim.

{\flushleft\bf Claim 2.} $u\notin C_1$.

Otherwise, let $u\in C_1$. Combining this with $\delta\leq h$ and Lemma 3.5, then
\begin{align*}
e(G')&=\sum_{i=1}^r e(C_i)+\sum_{1\leq i<j\leq r}e(C_i,C_j)=\sum_{i=1}^r e(C_i)+\lambda_r^h\\
&\leq\binom{n-(r-1)(h+1)-1}{2}+(r-1)\binom{h+1}{2}+\delta+\lambda_r^h\\
&=\frac{\big(n-(r-1)(h+1)-1\big)\big(n-(r-1)(h+1)-2\big)}{2}+\frac{(r-1)(h+1)h}{2}+\delta+\lambda_r^h\\
&<\frac{\big(n-(r-1)(h+1)-1\big)\big(n-(r-1)(h+1)-2\big)}{2}+\frac{(h+1)(h+2)}{2}+\\
&~~~~\frac{(r-3)(h+1)h}{2}+\frac{h(h+1)}{2}+\delta+\lambda_r^h.\\
\end{align*}
Repeat the above process of Claim 1,
\begin{align*}
\rho(G')& \leq \frac{\delta-1}{2}+\sqrt{2e(G')-n\delta+\frac{(\delta+1)^2}{4}}< n-(r-1)(h+1)-1,
\end{align*}
which contradicts (\ref{equ::21}). Thus, $u\notin C_1$.

Without loss of generality, assume that $u\in C_r$. By Claims 1, 2 and the maximality of $\rho(G')$,  we can obtain $G'[C_1]
\cong K_{n-(r-1)(h+1)}$, $G'[C_i] \cong K_{h+1}$ for $2\leq i\leq r-1$ and $G'[C_r\backslash u]\cong K_{h}$. Furthermore,
$\sum _{1\leq i<j\leq
r}e(C_i,C_j)=\lambda^h_r$.

{\flushleft\bf Claim 3.} $e(C_i,C_j)=0$ for $2\leq i<j\leq r$.

Otherwise, there exist some $i$ and $j$ $(2\leq i<j\leq r)$ such that $e(C_i,C_j)>0$. Suppose that $V(C_1)=\{v_1, v_2, \ldots, v_{n_1}\}$ and $V(C_i)=\{u_i^1,u_i^2,\ldots,u_i^{h+1}\}$ for $2 \leq i \leq r$.
Let $x$ be the Perron vector of $A(G')$. Without loss of generality, suppose that $x_{u_2^1}=\max\{x_v\mid v\in V(G')\backslash V(C_1)\}$, $x_{v_1} \geq x_{v_2} \geq \cdots \geq x_{v_{n_1}}$
 and $\sum _{3\leq i\leq r}e(C_2,C_i)=t$.  Then $t\leq \lambda_r^h-1$ and
\begin{align*}
\rho(G') x_{v_{n_1}} & \geq \sum_{i=1}^{n_1-1} x_{v_i}, \\
\rho(G') x_{u_2^1} & =\sum_{i=2}^r\sum_{v \in N_{C_i}(u_2^1)} x_v+\sum_{v \in N_{C_1}(u_2^1)} x_v \leq hx_{u_2^1}+t
x_{u_2^1}+\sum_{i=1}^{\lambda_r^h-t} x_{v_i}\\
&=(h+t) x_{u_2^1}+\sum_{i=1}^{\lambda_r^h-t} x_{v_i},
\end{align*}
and hence,
\begin{align*}
(\rho(G')-h-t)(x_{v_{n_1}}-x_{u_2^1}) \geq \sum_{i=1}^{n_1-h-t-1} x_{v_i}-\sum_{i=1}^{\lambda_r^h-t}
x_{v_i}=\sum_{i=\lambda_r^h-t+1}^{n_1-h-t-1} x_{v_i}>0
\end{align*}
due to $n \geq (\lambda_r^h+1)(h+1)^2$ and $\lambda_r^h\geq r-1$. Since $K_{n-(r-1)(h+1)}\subset G'$, it follows that
\begin{align*}
\rho(G')&>\rho(K_{n-(r-1)(h+1)})\\
&=n-(r-1)(h+1)-1 \\
&\geq (\lambda_r^h+1)(h+1)^2-\lambda_r^h(h+1)-1\\
&= (h(\lambda_r^h+1)+1)(h+1)-1\\
&>\lambda_r^h+h\\
&> h+t,
\end{align*}
where the inequalities follow from the fact that $n \geq (\lambda_r^h+1)(h+1)^2$, $\delta\leq h$ and $\lambda_r^h \geq r-1$, and hence
$x_{v_{n_1}}>x_{u_2^1}$. This suggests that $x_{v_i}>x_{u_2^1}$ for $1 \leq i \leq n_1$. We assert that $e(C_i,C_j)=0$ for
all $2\leq i<j\leq r$. Otherwise, there exists $vw\in E(G')$ such that $v\notin C_1$ and $w\notin C_1$. Take $G^*=G'-\{vw\}+\{v_1
v\}$. Then $G^*\in \mathcal{B}_{n,\delta}^{\lambda_r^h}$. Notice that $x_{v_1}
\geq x_v$ for $v \in V(G') \backslash\{v_1\}$. Thus, $\rho(G^*)>\rho(G')$ by Lemma \ref{lem::2.1}, a
contradiction. This implies that $e(C_i,C_j)=0$ for $2\leq i<j\leq r$.

Following this, we will prove that $G' \cong B_n^{\lambda_r^h, \delta}$. If $\lambda_r^h=r-1$, then
$e(C_1,C_i)=1$ for $2\leq i\leq r$ due to $G'$ is connected and $e(C_i,C_j)=0$ for all $2\leq i<j\leq r$. Suppose that
$N_{C_i}(C_1)=\{u_i^1\}$ for $2\leq i\leq r$. Then we assert that $N_{C_1}(u_i^1)=\{v_1\}$ for $2 \leq i \leq r$. Otherwise,  there exists some $v_p$ such that $v_pu_k^1\in
E(G')$, where $2\leq p\leq n_1$. Take $G^{**}=G'-\{v_pu_k^1\}+\{v_1 u_k^1\}$. Then $G^{**}\in \mathcal{B}_{n,\delta}^{\lambda_r^h}$. Notice that $x_{v_1} \geq x_v$ for $v \in V(G')
\backslash\{v_1\}$. Thus, $\rho(G^{**})>\rho(G')$ by Lemma \ref{lem::2.1}, a contradiction. Recall that
$u\in V(C_r)$. Similar to the proof of Lemma \ref{lem::3.4}, we can obtain $u\neq u_r^1$. Thus, $G \cong
B_n^{\lambda_r^h, \delta}$. We consider $\lambda_r^h\geq r$ in the following.  Recall that $N_{G'}(C_i)\subset
V(C_1)$ for $2\leq i\leq r$. Therefore, $G' \in \mathcal{K}_{n, \delta}^{\lambda_r^h}$. Combining this with Lemma
\ref{lem::3.4}, we can deduce that $G' \cong B_n^{\lambda_r^h, \delta}$, as required.

This completes the proof.

\end{proof}

\section{ Concluding remark}
In this paper, we have determined that $B_n^{\lambda_r^h, \delta}$ is the unique graph attaining the maximum spectral radius in $\mathcal{B}_{n,\delta}^{\lambda_r^h}$ for $n\geq (\lambda_r^h+1)(h+1)^2$ and $h\geq \delta$. For the case of $h<\delta$, we will consider the following question in the further work.

\begin{prob}\label{prob2}
What are the corresponding extremal graphs in $\mathcal{B}_{n,\delta}^{\lambda_r^h}$ with the maximum spectral radius of order $n\geq (\lambda_r^h+1)(h+1)^2$ and $h<\delta$?
\end{prob}

\section{ Acknowledgement}
This work is supported by the National Natural Science Foundation of China (Nos.12271162, 12326372, 12361071), and Natural Science Foundation of Shanghai (Nos. 22ZR1416300 and 23JC1401500), and The Program for Professor of Special Appointment (Eastern Scholar) at Shanghai Institutions of Higher Learning (No. TP2022031) and  Excellent Doctor Innovation program of Xinjiang University (No. XJU2024BS043).



\end{document}